\documentclass[11pt]{amsart}
\usepackage{amssymb,amsmath}
\usepackage[all]{xy}
\usepackage{mathrsfs}

\newtheorem{thm}{Theorem}[section]
\newtheorem{lemma}[thm]{Lemma}
\newtheorem{prop}[thm]{Proposition}
\newtheorem{cor}[thm]{Corollary}

\newtheorem{Th}{Theorem}

\theoremstyle{definition}
\newtheorem{defn}[thm]{Definition}
\newtheorem{nota}[thm]{Notation}
\newtheorem{rem}[thm]{Remark}

\newcommand{\ls}[2]{{^{#1}\!{#2}}}
\newcommand{\qbox}[1]{\hspace{.2cm}\hbox{#1}\hspace{.2cm}}
\newcommand{\ovl}[1]{\overline{#1}}

\def\car{\trianglelefteq_{\operatorname{char}}}

\def\proj{(\operatorname{proj})}
\def\conj{\operatorname{conj}}

\def\Res{\operatorname{Res}\nolimits}
\def\Hom{\operatorname{Hom}\nolimits}
\def\End{\operatorname{End}\nolimits}
\def\Aut{\operatorname{Aut}\nolimits}
\def\Out{\operatorname{Out}\nolimits}

\def\Syl{\operatorname{Syl}\nolimits}
\def\GL{\operatorname{GL}\nolimits}

\def\Z{\mathbb Z}
\def\F{\mathbb F}
\def\CA{\mathcal A}

\def\CO{\mathscr O}

\title{On the orbit category on nontrivial $p$-subgroups and endotrivial modules}
\author{Nadia Mazza}
\address
{School of Mathematical Sciences\\ Lancaster University\\ Lancaster \\LA1 4YF, UK} 
\email{n.mazza@lancaster.ac.uk}

\date{\today}

\begin{document}

\begin{abstract} 
Let $p$ be a prime, let $G$ be a finite group of order divisible by $p$, and let $k$ be a field of characteristic $p$.
An endotrivial $kG$-module is a finitely generated $kG$-module $M$ such that its endomorphism algebra $\End_kM$ decomposes as the direct sum of a one-dimensional trivial $kG$-module and a projective $kG$-module.
In this article, we determine the fundamental group of the orbit category on nontrivial $p$-subgroups of $G$ for a large class of finite groups, and use Grodal's approach to describe the group of endotrivial modules for such groups.
Hence, we improve on the results about the group of endotrivial modules for finite groups with abelian Sylow $p$-subgroups obtained by Carlson and Th\'evenaz.
With some additional analysis, we then determine the fundamental group of the orbit category on nontrivial $p$-subgroups of $G$ and the group of endotrivial $kG$-modules in the case when $G$ has a metacyclic Sylow $p$-subgroup for $p$ odd.
\end{abstract}

\maketitle

\section{Introduction}\label{sec:intro}
Let $p$ be a prime, let $G$ be a finite group with order divisible by $p$, and let $k$ be a field of characteristic $p$.
It is well-known that the study of the modular representations of $G$ over $k$ in full generality is hopeless; however being able to describe certain classes of $kG$-modules can be feasible and reveal some intriguing relationships which lead to intradisciplinary applications. The endotrivial modules are such `useful' modules.
Originally introduced as {\em invertible modules} by J. Alperin in \cite{Alp0}, the term {\em endotrivial} (or rather {\em endo-trivial}) was coined by E. Dade in his pioneering work on endopermutation modules for finite $p$-groups in \cite{Da}.
Loosely, Alperin's approach presented endotrivial modules as the elements whose isomorphism classes generate the unit group of the stable module category of $G$, whereas Dade's viewpoint was an in-depth analysis of the structure of the sources of simple modules for a large class of finite groups.
In the last 15 years, the various studies of endotrivial modules have shown that ``they are important"; though their classification remains an open problem, see \cite{B,CT,Gr,Ma,NR} for instance.
 
In this article, we use Grodal's approach via homotopy theory (cf. \cite{Gr}) to describe the group of endotrivial modules for a large class of finite groups, comprising those with an abelian Sylow $p$-subgroup, and hence we determine the fundamental group of the orbit category on nontrivial $p$-subgroups for these groups, thus improving on the results obtained in \cite{CT}.
Some additional analysis enables us to do the same for a finite group with a split metacyclic Sylow $p$-subgroup for $p$ odd.

We now state our two main results (see Notation~\ref{ssec:nota} and Section~\ref{sec:et} for the notation).
We will recall how we compute $\pi_1(\CO^*_p(G))$ in Section \ref{sec:et}.

\begin{Th}\label{thm:main}
Let $p$ be a prime, let $G$ be a finite group with order divisible by $p$ and let $S$ be a Sylow $p$-subgroup of $G$.
Suppose that $\Omega_1(S)\leq Z(S)$ and that $N_G(S)$ controls $p$-fusion in $G$. 
Then
$$\pi_1(\CO^*_p(G))\cong N_G(S)/\langle N_G(S)\cap O^{p'}(N_G(Q))\mid1<Q\leq\Omega_1(S)\rangle.$$
It follows that
$$K(G)\cong\Hom(N_G(S)/J,k^\times)\qbox{and}J=\langle N_G(S)\cap SN_G(Q)'\mid1<Q\leq\Omega_1(S)\rangle.$$
\end{Th}
Note that Theorem~\ref{thm:main} generalises the results of \cite{CT}, since finite groups with abelian Sylow $p$-subgroups satisfy the hypotheses. 
As an immediate corollary, shown in Section~\ref{sec:et}, we obtain the description of the group of endotrivial $kG$-modules for this class of finite groups.

The above result, together with known results and some additional analysis yield our second main result.

\begin{Th}\label{thm:main2}
Let $p$ be an odd prime, let $G$ be a finite group with order divisible by $p$, let $k$ be a field of characteristic $p$ and let $S$ be a Sylow $p$-subgroup of $G$. 
Suppose that $S$ is metacyclic and let $E=\Omega_1(S)$.
The following holds.
\begin{enumerate}
\item If $S$ is cyclic, then the restriction map $\Res^G_{N_G(E)}:~T(G)\longrightarrow T(N_G(E))$ is an isomorphism, and $T(N_G(E))$ is a finite group, given by an extension
$$\xymatrix{0\ar[r]&K(N_G(E))\ar[r]&T(N_G(E))\ar[r]&\langle[\Omega(k)]\rangle\ar[r]&0},$$ 
where $K(N_G(E))\cong\Hom(N_G(E),k^\times)$.
\item If $S$ is nonabelian nonsplit and if $G$ has no proper strongly $p$-embedded subgroup, then $T(G)=\langle[\Omega(k)]\rangle\oplus K(G)$, with $K(G)\cong\Hom(G,k^\times)$ and $\langle[\Omega(k)]\rangle\cong\Z$.
\item If $S$ is split and if $G$ has no proper strongly $p$-embedded subgroup, then $T(G)=\langle[\Omega(k)]\rangle\oplus K(G)$, where $K(G)\cong\Hom(N_G(S)/R,k^\times)$ and:
\begin{enumerate}
\item if $E\leq Z(S)$, then $R=\langle N_G(S)\cap O^{p'}(N_G(P))\mid1<P\leq E\rangle$, whilst 
\item if $E\not\leq Z(S)$, then $R$ is generated by the elements of $N_G(S)\cap O^{p'}(N_G(Z))$ and those of $N_G(S)$ that can be written as a product of the form 
$g=hh't$, with with $h\in O^{p'}(N_G(E))\cap C_G(Q)$, $h'\in O^{p'}(C_G(Q))$ and $t\in S$, where $Q$ is a noncentral subgroup of $E$. 
\end{enumerate}
In particular, if $C_G(S)\leq R$, then $K(G)$ is cyclic of order dividing $p-1$.
\end{enumerate}
\end{Th}
The first case was shown in \cite{MT} and revisited in \cite{CT}. The second case is an immediate consequence of an observation by J. Dietz in \cite{Di} and the results of \cite{CMT,NR}. 
The assumption that $G$ does not contain a proper strongly $p$-embedded subgroup is a very mild one, see Theorem~\ref{thm:e-t}~\eqref{it:spe}.
The actual novelty is the case of a finite group with a nonabelian split metacyclic Sylow $p$-subgroup. 
Observe that if $E\leq Z(S)$, this is a corollary of Theorem~\ref{thm:main}, leaving us to consider the case when $E\not\leq Z(S)$ in order to complete the proof of Theorem~\ref{thm:main2}, which we achieve in Proposition~\ref{prop:main}. 

\vspace{.3cm}
\noindent{\bf Proposition~\ref{prop:main}.}\;{\em Let $p$ be an odd prime, let $G$ be a finite group with order divisible by $p$ and let $S\in\Syl_p(G)$. Suppose that $S$ is split metacyclic and that $G$ has no strongly $p$-embedded subgroup. Let $E=\Omega_1(S)$.}
\begin{enumerate}
\item {\em If $E\leq Z(S)$, then $\pi_1(\CO^*_p(G))\cong N_G(S)/\langle N_G(S)\cap O^{p'}(N_G(P))\mid1<P\leq E\rangle$.}
\item {\em If $E\not\leq Z(S)$, then $\pi_1(\CO^*_p(G))\cong N_G(S)/R$, where $R$ is generated by the elements of $N_G(S)\cap O^{p'}(N_G(Z))$ and those of $N_G(S)$ that can be written as a product of the form 
$g=hh't$, with with $h\in O^{p'}(N_G(E))\cap C_G(Q)$, $h'\in O^{p'}(C_G(Q))$ and $t\in S$. }
\end{enumerate}

\vspace{.3cm}

Let us emphasise that the case $p=2$ requires a different method. 
To motivate this statement, recall that if $S$ is a metacyclic $2$-group, then the elements of order $2$ generate a dihedral $2$-subgroup of $S$, nonabelian in general, whereas if $p$ is odd, then the elements of order $p$ generate an elementary abelian $p$-subgroup of $S$ which is characteristic in $S$. 
We refer to \cite{CMT2} for specific metacyclic $2$-groups (cf. also \cite{Di2}).

\vspace{0.3cm}
The paper is organised as follows. In Section~\ref{sec:et} we review the known results on endotrivial modules that we need, and prove Theorem~\ref{thm:main}. 
In Section~\ref{sec:mc}, we revisit the structure and properties of finite groups with metacyclic Sylow $p$-subgroups.
Finally, in Section~\ref{sec:et-mc}, we combine the facts of the previous two sections to complete the proofs of Proposition~\ref{prop:main} and hence Theorem~\ref{thm:main2}.

\subsection{Notation}\label{ssec:nota}
Given a group $G$ and a prime $p$, $\Syl_p(G)$ denotes the set of Sylow $p$-subgroups of $G$ and $G'=[G,G]$ the derived subgroup. 
We let $O^{p'}(G)$ be the smallest normal subgroup of $G$ of $p'$-index, that is, $O^{p'}(G)=\langle S\mid S\in\Syl_p(G)\rangle$ is the subgroup of $G$ generated by its $p$-elements.
Let $\CO^*_p(G)$ denote the orbit category on nontrivial $p$-subgroups of $G$. 
This is the category with objects the transitive $G$-sets $G/P$, where $P$ runs through the nontrivial $p$-subgroups of $G$, 
and the morphisms are given by the $G$-equivariant maps, $\xymatrix{G/P\ar[r]^{[g]}&G/Q}$, for $g\in G$ such that $P\leq\ls gQ$, where $\ls gQ=gQg^{-1}$ and $[g](xP)=xgQ$ for all $xP\in G/P$ (cf. \cite[Section 2.4]{Gr}). 
A subgroup $H$ of $G$ controls $p$-fusion in $G$ if $H$ contains a Sylow $p$-subgroup $S$ of $G$ and whenever $Q,gQg^{-1}\leq S$ for some $g\in G$, then there exist $h\in H$ and $c\in C_G(Q)$ such that $g=hc$.
A subgroup $H$ of $G$ is strongly $p$-embedded in $G$ if $p\mid |H|$ and $p\nmid|H\cap\ls gH|$ for all $g\in G-H$.
The smallest strongly $p$-embedded subgroup of $G$ containing $S\in\Syl_p(G)$ is $G_0=\langle N_G(Q)\mid1<Q\leq S\rangle$, cf. \cite[Section 46]{As}. 
If $S$ is a finite $p$-group, then $\Phi(S)=S'\langle g^p\mid g\in S\rangle$ is the Frattini subgroup and $\Omega_1(S)=\langle g\in S\mid g^p=1\rangle$. 
We write $\Hom(G,k^\times)$ for the group of $1$-dimensional $kG$-modules, and $k$ denotes both: the field and the trivial $1$-dimensional module.
Less common notation will be introduced as and when needed below. We refer to \cite{As,G,huppert} for additional background.

We end Section~\ref{sec:intro} with two well-known results in group theory which we will use liberally.
\begin{lemma}\cite[Dedekind's identity Hilfsatz I.2.12 c)]{huppert}
Let $G$ be a group and let $A,B,C\leq G$ with $A\leq C$. Then $A(B\cap C)=AB\cap C$.
\end{lemma}

\begin{proof}
The inclusion $\subseteq$ is clear. Conversely, let $a\in A$ and $b\in B$ such that $ab=c\in AB\cap C$. Then, $b=a^{-1}c\in C$, and so $ab\in A(B\cap C)$.
\end{proof}

\begin{lemma}\cite[Frattini's argument Theorem I.3.7]{G}\label{lem:frattini}
Let $G$ be a finite group and let $H$ be a normal subgroup of $G$ of order divisible by $p$.
Then $G=HN_G(S)$, where $S\in\Syl_p(H)$.  
\end{lemma}

\begin{proof}
For any $g\in G$, we have $\ls gS\leq H$ since $H\trianglelefteq G$. Thus, there exists $h\in H$ such that $\ls gS=\ls hS$, and it follows that $h^{-1}g\in N_G(S)$, as required.
\end{proof}

\section{Endotrivial modules}\label{sec:et}

Let $p$ be a prime, let $G$ be a finite group with order divisible by $p$, and let $k$ be a field of characteristic $p$.
All the modules we consider are finitely generated. If $M$ is a $kG$-module, we write $M^*=\Hom_k(M,k)$ for its $k$-dual and $\End_kM$ for the set of $k$-linear maps $M\to M$, both regarded as $kG$-modules. 
We write $\otimes=\otimes_k$, and we recall that $M^*\otimes M\cong\End_kM$ as $kG$-modules.
We refer to \cite{Ma} for additional background in modular representation theory.

\begin{defn}\label{def:e-t}
An {\em endotrivial $kG$-module} is a finitely generated $kG$-module $M$ such that $M^*\otimes M\cong k\oplus\proj$ as $kG$-modules, where $\proj$ denotes a projective $kG$-module. 
Equivalently, an endotrivial module is an invertible element in the stable module category of finitely generated $kG$-modules.

The {\em group of endotrivial modules} of $G$ is the set $T(G)$ of stable isomorphism classes of endotrivial $kG$-modules, with the binary operation $[M]+[N]=[M\otimes N]$. 
\end{defn}
Note that the structure of $T(G)$ depends on $k$.
Recall that two $kG$-modules $M,N$ are isomorphic in the stable module category if there exist projective $kG$-modules $U,V$ such that $M\oplus U\cong N\oplus V$.

The description of $T(G)$ invokes the $G$-orbit space $\CA^{\geq2}_p(G)/G$, defined as follows. 
The elements of the poset $\CA^{\geq2}_p(G)$ are the noncyclic elementary abelian $p$-subgroups of $G$ and the order relation is the subgroup inclusion. 
The group $G$ acts by conjugation on $\CA^{\geq2}_p(G)$, and we write $\CA^{\geq2}_p(G)/G$ for the set of $G$-orbits. 
The induced poset structure on $\CA^{\geq2}_p(G)/G$ allows us to split it into components, whereby to orbits $\{\ls GA\}$ and $\{\ls GB\}$ lie in the same component if and only we can pick representatives that are linked by a chain of inclusions and containments up to conjugation, $A\leq\ls{g_1}A_1\geq\cdots\leq\ls{g_n}A_n\geq B$, for some $A_1,\dots,A_n\in\CA^{\geq2}_p(G)$ and $g_1,\dots,g_n\in G$ (cf.\cite{Ma-poset}).

We summarise the known results about the structure of $T(G)$ in Theorem~\ref{thm:e-t}. The proofs (or references to them) can be found in \cite{CMT,Ma,NR}. 

\begin{thm}\label{thm:e-t}
Let $p$ be a prime, let $G$ be a finite group with order divisible by $p$, and let $k$ be a field of characteristic $p$. Let $S\in\Syl_p(G)$.
\begin{enumerate}
\item $T(G)$ is a finitely generated abelian group. Let us write $T(G)=TT(G)\oplus TF(G)$, where the torsion subgroup $TT(G)$ is finite and $TF(G)$ is torsionfree of finite $\Z$-rank.
\item $TT(S)=0$ unless $S$ is cyclic of order at least $3$, generalised quaternion or semidihedral (these last two cases for $p=2$ only).
\item $T(N_G(S))\cong\Hom(N_G(S),k^\times)+T(S)^{N_G(S)}$, where 
$$T(S)^{N_G(S)}=\{[M]\in T(S)\mid M\cong\ls gM\;\hbox{as $kS$-modules}\quad\forall\;g\in N_G(S)\}$$
is the group of $N_G(S)$-stable endotrivial $kS$-modules.
\item\label{it:spe} The restriction $\Res^G_{N_G(S)}:T(G)\to T(N_G(S))$ is injective, and $\Res^G_H:T(G)\to T(H)$ is an isomorphism if $H$ is a strongly $p$-embedded subgroup of $G$. 
\item $\ker(\Res^G_S:T(G)\to T(S))\cong\Hom(G,k^\times)$ whenever $S\cap\ls gS$ is nontrivial for all $g\in G$.
\item If $G$ is $p$-soluble (cf. \cite[Chapter 6]{G}) of $p$-rank at least $2$, then $TT(G)\cong\Hom(G,k^\times)$.
\item The rank of $TF(G)\cong T(G)/TT(G)$ as abelian group is equal to the number of connected components of the $G$-orbit space $\CA^{\geq2}_p(G)/G$. 
Moreover, if $\CA^{\geq2}_p(G)/G$ is connected, then $T(G)=TT(G)\oplus\langle[\Omega(k)]\rangle$, where $\Omega(k)$ is the kernel of a projective presentation of the trivial module $k$.
\end{enumerate}
\end{thm}

As a corollary of Theorem~\ref{thm:e-t}, we obtain that if $S$ is neither cyclic of order at least $3$, generalised quaternion or semidihedral, then 
$$TT(G)=K(G)=\ker(\Res^G_S:T(G)\to T(S)).$$
That is, $TT(G)$ is generated by the stable isomorphism classes of the trivial source endotrivial $kG$-modules, i.e. the indecomposable endotrivial $kG$-modules $M$ such that $\Res^G_SM\cong k\oplus\proj$.

\vspace{.3cm}
The group of endotrivial modules of a finite group $G$ has been determined for several families of groups.
In this article we revisit some of these indirectly, since finite groups with metacyclic Sylow $p$-subgroups comprise those with cyclic Sylow $p$-subgroups \cite[Theorem 3.6]{MT}, those with dihedral, semidihedral or quaternion Sylow $2$-subgroups \cite{CMT2}, and in some cases such groups are $p$-soluble \cite[Theorem 6.2]{CMT} and \cite{NR}. 
The main obstructions to determine $T(G)$ in general are:
\begin{itemize}
\item[(i)] Determine $K(G)$. 
\item[(ii)] Find generators for a torsionfree subgroup $TF(G)$ if $T(G)/TT(G)$ is not cyclic.
\end{itemize} 

\vspace{.3cm}
The study of endotrivial modules has inspired a variety of methods, mostly concerned with the study of $K(G)$. 
Indeed, the investigations around trivial source endotrivial modules lead to surprising connections with other areas of mathematics.

In \cite{B}, Balmer shows that $K(G)$ is isomorphic to the group of weak $S$-homomorphisms of $G$ (over $k$), where $S\in\Syl_p(G)$. A {\em weak $S$-homomorphism of $G$} is a function $\chi:G\to k^\times$ such that:
\begin{itemize}
\item $\chi(g)=1$ for all $g\in S$ and for all $g\in G$ such that $S\cap\ls gS=1$.
\item For any $g,h\in G$ such that $S\cap\ls gS\cap\ls{gh}S\neq1$, then $\chi(gh)=\chi(g)\chi(h)$.
\end{itemize}
Multiplication of two weak $S$-homomorphisms is pointwise. 
Using Balmer's work, Carlson and Th\'evenaz show in \cite[Theorem 5.1]{CT} that if $S$ is abelian, then the group of weak $S$-homomorphisms of $G$ is isomorphic to $\Hom(N_G(S)/\rho^2(S),k^\times)$, 
where $\rho^2(S)=\langle N_G(S)\cap\rho^1(P)\mid\;1<P\leq S\rangle$. 
More generally, they associate to every nontrivial subgroup $P$ of $S$ an ascending sequence of subgroups $(\rho^i(P))_{i\geq1}$ of $N_G(P)$ defined inductively as follows~\cite[Section 4]{CT}.
\begin{itemize}
\item $\rho^1(P)=S_PN_G(P)'$, where $S_P\in\Syl_p(N_G(P))$.
\item $\rho^{i+1}(P)=\langle N_G(P)\cap\rho^i(Q)\mid1<Q\leq S\rangle$ for all $i\geq1$.
\end{itemize}

In \cite{CT}, the authors hint at some properties of these subgroups, which we spell out.
\begin{rem}\label{rem:rho-sgps}
\begin{enumerate}
\item Since $\Omega_1(P)\car P$ for any $p$-group $P$, we have $\rho^1(Q)\leq\rho^1(\Omega_1(Q))$ for all $Q\leq S$, and, arguing inductively on $i\geq1$, 
we observe that 
$$\rho^i(Q)=\langle N_G(Q)\cap\rho^{i-1}(P)\mid1<P\leq S\rangle=\langle N_G(Q)\cap\rho^{i-1}(P)\mid1<P\leq\Omega_1(S)\rangle$$ 
for any nontrivial subgroup $Q$ of $S$. Hence $\rho^\infty(S)$ is attained using the nontrivial subgroups of $\Omega_1(S)$.
\item If there exists a positive integer $i$ such that $\rho^i(Q)=\rho^{i+1}(Q)$, for all $1<Q\leq\Omega_1(S)$, then $\rho^i(P)=\rho^\infty(P)$ for all $1<P\leq S$. 
\end{enumerate}
\end{rem}

A more recent article of Grodal uses homotopy theory to identify $K(G)$ as the abelianisation of the fundamental group of the orbit category on nontrivial $p$-subgroups of $G$. 
In \cite[Addendum 3.4]{Gr}, Grodal relates his result with Balmer's and in \cite[Theorem F and Section 5]{Gr}, he shows the Carlson-Th\'evenaz conjecture asserting that $K(G)\cong N_G(S)/\rho^\infty(S)$.

\begin{thm}\cite[Theorems A and 4.10]{Gr}\label{thm:gr}
Let $\CO^*_p(G)$ denote the orbit category on nontrivial $p$-subgroups of $G$.
Then $K(G)$ is isomorphic to the group of isomorphism classes of functors $\CO^*_p(G)\longrightarrow k^\times$, where $k^\times$ is regarded as a category with one object. 
Equivalently, in terms of group theory, let $R$ denote the subgroup of $N_G(S)$ generated by all the elements $g$ such that there exist nontrivial subgroups $Q_0,\dots,Q_n\leq S$ and elements $g_i\in O^{p'}(N_G(Q_i))$ for all $1\leq i\leq n$ with
\begin{itemize}
\item $g=g_1\cdots g_n$, and 
\item $Q_0^{g_1\cdots g_i}\leq Q_{i+1}$, for all $1\leq i\leq n-1$ (the initial condition being $Q_0\leq Q_1$).
\end{itemize}
Then $\pi_1(\CO^*_p(G))\cong N_G(S)/R$ and thus $K(G)\cong\Hom(N_G(S)/R,k^\times)$.
\end{thm}

Theorem~\ref{thm:gr} can be visualised as follows:  
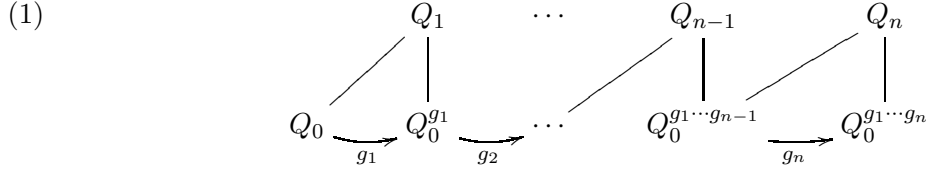
\begin{equation}\label{eq:gr}
\xymatrix{&Q_1\ar@{-}[dl]\ar@{-}[d]&\cdots&Q_{n-1}\ar@{-}[dl]\ar@{-}[d]&Q_n\ar@{-}[dl]\ar@{-}[d]\\
Q_0\ar@/_/[r]_{g_1}&Q_0^{g_1}\ar@/_/[r]_{g_2}&\cdots&Q_0^{g_1\cdots g_{n-1}}\ar@/_/[r]_{g_n}&Q_0^{g_1\cdots g_n}}
\end{equation}
where the arrows mean right conjugation and the undirected edges mean that what is above contains what is below. 

All is now in place to prove Theorem~\ref{thm:main}.

\vspace{.3cm}
\noindent{\bf Theorem~\ref{thm:main}.}
{\em Let $p$ be a prime, let $G$ be a finite group with order divisible by $p$ and let $S$ be a Sylow $p$-subgroup of $G$.
Suppose that $\Omega_1(S)\leq Z(S)$ and that $N_G(S)$ controls $p$-fusion in $G$. 
Then
$$\pi_1(\CO^*_p(G))\cong N_G(S)/\langle N_G(S)\cap O^{p'}(N_G(Q))\mid1<Q\leq\Omega_1(S)\rangle.$$
It follows that
$$K(G)\cong\Hom(N_G(S)/J,k^\times),\qbox{with}J=\langle N_G(S)\cap SN_G(Q)'\mid1<Q\leq\Omega_1(S)\rangle.$$}

\vspace{.3cm}
We know that the result holds if $S$ is abelian.
Moreover, if $S$ is cyclic, then $H=N_G(\Omega_1(S))$ is strongly $p$-embedded in $G$ and $\pi_1(\CO^*_p(G))\cong\pi_1(\CO^*_p(H))\cong H/O^{p'}(H)$. 
By Frattini's argument, $H=N_H(S)O^{p'}(H)$, which implies that $H/O^{p'}(H)\cong N_H(S)/(N_H(S)\cap O^{p'}(H))$.
Moreover, the abelianisation of $H/O^{p'}(H)$ is $H/H'O^{p'}(H)=H/H'S$, the largest abelian $p'$-quotient of $H$.

\begin{proof}
For convenience, let $N=N_G(S)$, and, for every $1<Q<S$, let $N_Q=N_G(Q)$, let $C_Q=C_G(Q)$, and choose $S_Q\in\Syl_p(N_Q)$. 

We have
$\langle N\cap O^{p'}(N_Q)\mid1<Q\leq S\rangle=\langle N\cap O^{p'}(N_Q)\mid1<Q\leq\Omega_1(S)\rangle$ and similarly 
$\langle N\cap S_QN_Q'\mid1<Q\leq S\rangle=\langle N\cap SN_Q'\mid1<Q\leq\Omega_1(S)\rangle$.
Consequently, we can choose the subgroups $Q_0,\dots,Q_n$ of Theorem~\ref{thm:gr} in $\Omega_1(S)$.
Let $R$ be the subgroup defined in Theorem~\ref{thm:gr}.
We want to show that for every $g\in R$ there exists a nontrivial subgroup $Q$ of $\Omega_1(S)$ such that $g\in O^{p'}(N_Q)$.
Let $g\in R$, and say that we can factorise $g=g_1\cdots g_n$ with $g_i\in O^{p'}(N_{Q_i})$ such that $Q_0,\dots,Q_n$ are subject to the conditions of Theorem~\ref{thm:gr}.
In particular, $Q_0,Q_0^{g_1}\leq Q_1$ with $g_1\in O^{p'}(N_{Q_1})$. 

Since $N$ controls $p$-fusion in $G$, we have $N_Q=(N\cap N_Q)C_Q$, with $O^{p'}(N_Q)=\langle\ls gS\mid g\in C_Q\rangle$, for all $1<Q\leq\Omega_1(S)$, since $\Omega_1(S)\leq Z(S)$.
Note that if $1<Q_0\leq Q_1\leq\Omega_1(S)$, then $C_{Q_1}\leq C_{Q_0}$ and so $O^{p'}(N_{Q_1})\leq O^{p'}(N_{Q_0})$.
We conclude that $Q_0=Q_0^{g_1}$. Proceeding inductively, we obtain $g_1\cdots g_i\in O^{p'}(N_{Q_0})$, for all $1\leq i\leq n-1$, and so $g\in O^{p'}(N_{Q_0})$, for some $1<Q_0\leq\Omega_1(S)$.
This holds for all $g\in R$, proving that $R=\langle N\cap O^{p'}(N_Q)\mid1<Q\leq\Omega_1(S)\rangle$. 
By Theorem~\ref{thm:gr}, we have $\pi_1(\CO^*_p(G))\cong N/R=N/\langle N\cap O^{p'}(N_Q)\mid1<Q\leq\Omega_1(S)\rangle$.

It remains to show that the abelianisation $(N/R)/[N/R,N/R]=N/N'R$ of $N/R$ is isomorphic to $N/J$, i.e. that $N'R=J$, where $J=\langle N_G(S)\cap SN_G(Q)'\mid1<Q\leq\Omega_1(S)\rangle$ (this should be well known, but we were unable to find a reference). 

For all $1<Q<S$, we have $O^{p'}(N_Q)\leq S_QN_Q'$, from which we deduce the inclusion $R\leq J$.
Since $N/J$ is abelian we have $N'\leq J$, and thus $N'R\leq J$.

Conversely, by Frattini's argument, we can write $N_Q=(N\cap N_Q)O^{p'}(N_Q)$, for all $1<Q\trianglelefteq S$.
Then we calculate, for all $1<Q\leq\Omega_1(S)$, 
$$SN_Q'=S(N\cap N_Q)'O^{p'}(N_Q)'[O^{p'}(N_Q),N\cap N_Q]\leq(N\cap N_Q)'O^{p'}(N_Q)$$
(using that $N\cap N_Q$ normalises $O^{p'}(N_Q)$). 
Hence, for all $1<Q\leq\Omega_1(S)$,
$$N\cap SN_Q'\leq N\cap(N\cap N_Q)'O^{p'}(N_Q)=(N\cap N_Q)'\big(N\cap O^{p'}(N_Q)\big),$$
where we have applied Dedekind's identity. 
Therefore, $N\cap SN_Q'\leq N'\big(N\cap O^{p'}(N_Q)\big)$.
Finally, note that $N'R=\langle N'\big(N\cap O^{p'}(N_Q)\big)\mid1<Q\leq\Omega_1(S)\rangle$, because $N'\trianglelefteq N$ and $R$ is generated by subgroups of $N$.
The equality $J=N'R$ follows.
\end{proof}

For any finite group $G$, the poset $\CA^{\geq2}_p(G)/G$ of the $G$-orbits of noncyclic elementary abelian $p$-subgroups of $G$ is connected whenever $Z(S)$ is noncyclic. 
Therefore, the difficulty in determining the group of endotrivial modules $T(G)$ over a group $G$ satisfying the hypothesis of Theorem~\ref{thm:main} is to determine the subgroup $K(G)$ of $T(G)$ generated by the trivial source endotrivial modules. 
Indeed, let $S\in\Syl_p(G)$ and suppose that $\Omega_1(S)\leq Z(S)$. If $\Omega_1(S)$ is cyclic, then $S$ is cyclic, or possibly $p=2$ and $S$ is generalised quaternion (cf. \cite[Satz III.8.2]{huppert}). 
In this case, $\Res^G_{G_0}:T(G)\to T(G_0)$ is an isomorphism, where $G_0=N_G(\Omega_1(S))$ is strongly $p$-embedded in $G$. 
Furthermore, $T(G_0)$ is a finite group, described by an extension $\xymatrix{0\ar[r]&\Hom(G_0,k^\times)\ar[r]&T(G_0)\ar[r]&T(S)\ar[r]&0}$, see \cite[Theorem 3.2]{MT} and \cite[Lemma 4.1 and Theorem 4.5]{CMT2}. 
It remains to consider the case when $\Omega_1(S)\leq Z(S)$ is not cyclic, knowing that $T(G)=\langle[\Omega(k)]\rangle\oplus K(G)$ and $\langle[\Omega(k)]\rangle\cong\Z$. 

\begin{cor}\label{cor:main}
Let $p$ be a prime, let $G$ be a finite group with order divisible by $p$ and let $S$ be a Sylow $p$-subgroup of $G$.
Suppose that $\Omega_1(S)\leq Z(S)$ is not cyclic and that $N_G(S)$ controls $p$-fusion in $G$. 
Then $T(G)=\langle[\Omega(k)]\rangle\oplus K(G)$, with $\langle[\Omega(k)]\rangle\cong\Z$ and $K(G)\cong\Hom(N_G(S)/J,k^\times)$, where $J=\langle N_G(S)\cap SN_G(Q)'\mid1<Q\leq\Omega_1(S)\rangle$.
\end{cor}

\vspace{.3cm}
We end this section with the explicit construction of weak homomorphisms for a group satisfying the hypothesis of Theorem~\ref{thm:main}. 
This is similar to \cite[Theorems 5.1 and 7.1]{CT}, and somewhat implicit in that article.

\begin{prop}\label{prop:omega1central}
Let $p$ be a prime, let $G$ be a finite group with order divisible by $p$ and let $S\in\Syl_p(G)$.
Suppose that $\Omega_1(S)\leq Z(S)$ and that $N_G(S)$ controls $p$-fusion in $G$. 
Then, the group of weak $S$-homomorphisms of $G$ is isomorphic to the subgroup of $\Hom(N_G(S),k^\times)$ generated by the maps containing $\rho^2(S)$ in their kernel.
\end{prop}
Note that, in the hypotheses of Proposition~\ref{prop:omega1central}, $\Aut_G(\Omega_1(S))=N_G(\Omega_1(S))/C_G(\Omega_1(S))\cong N_G(S)/(N_G(S)\cap C_G(\Omega_1(S)))$ is isomorphic to a quotient of $\Out_G(S)= N_G(S)/SC_G(S)$.

\begin{proof}
For convenience, let $N=N_G(S)$, and, for every $1<Q<S$, let $N_Q=N_G(Q)$, let $C_Q=C_G(Q)$, and choose $S_Q\in\Syl_p(N_Q)$. 
Let $1<Q\leq\Omega_1(S)$. Then $S=S_Q\leq C_Q$.
Since $N$ controls $p$-fusion in $G$, we have $N_Q=(N\cap N_Q)C_Q\subseteq NC_Q=\{gc\mid g\in N,\;c\in C_Q\}$, for all $1<Q\leq\Omega_1(S)$.
By Frattini's argument, $C_Q=(N\cap C_Q)SC_Q'$, where $SC_Q'\trianglelefteq C_Q$. 
This yields $C_Q/SC_Q'\cong(N\cap C_Q)/(N\cap C_Q\cap SC_Q')=(N\cap C_Q)/(N\cap SC_Q')$.
Consequently (cf. \cite[Lemma 5.2]{CT}, there is a bijection 
$$\{\phi\in\Hom(N\cap C_Q,k^\times)\mid\ (N\cap SC_Q')\leq\ker(\phi)\}\leftrightarrow\{\psi\in\Hom(C_Q,k^\times)\mid\ SC_Q'\leq\ker(\psi)\}.\quad(*)$$

It now suffices to observe that the arguments used in the proof of \cite[Theorem 5.1]{CT} apply in this case too. 
Let $\chi\in\Hom(N,k^\times)$ with $\rho^2(S)\leq\ker(\chi)$.
Let $1<Q\leq S$. By Remark~\ref{rem:rho-sgps}, we may assume that $Q\leq\Omega_1(S)$.
Since $\chi|_{N\cap SC_Q'}=1$ is the trivial character, the bijection $(*)$ says that there exists a unique $\psi_Q\in\Hom(C_Q,k^\times)$ with $\psi_Q|_{N\cap C_Q}=\chi|_{N\cap C_Q}$.
Define $\theta:G\to k^\times$ as in \cite[Theorem 5.1]{CT}, that is:
\begin{itemize}
\item If $S\cap\ls gS=1$, put $\theta(g)=1$.
\item If $S\cap\ls gS\neq1$, then let $Q=\Omega_1(S\cap\ls gS)$ and write $g=cn$ with $c\in C_Q$ and $n\in N$. 
\end{itemize}
Note that any two such decompositions of $G$ differ by an element of $N\cap C_Q$, in the sense that $c'n'=cn$ if and only if $c^{-1}c'=nn'^{-1}\in N\cap C_Q$. 
Set $\theta(g)=\psi_Q(c)\chi(n)$.
From the equality $\psi_Q|_{N\cap C_Q}=\chi|_{N\cap C_Q}$, we deduce that $\theta(g)$ is independent of the choice of the decomposition $g=cn$.

A routine verification shows that $\theta$ is a weak $S$-homomorphism of $G$: Let $g,h\in G$.
\begin{itemize}
\item If $g\in S$, then $\theta(g)=\psi_{\Omega_1(S)}(1)\chi(g)=1$ since $S\leq\rho^2(S)\leq\ker(\chi)$. 
\item If $S\cap\ls gS=1$, then $\theta(g)=1$ by definition.
\item If $S\cap\ls gS\cap\ls{gh}S\neq1$, then let $Q=\Omega_1(S\cap\ls gS\cap\ls{gh}S)$. 
Since $1<Q\leq\Omega_1(S\cap\ls gS)$, we can write $g=cn$ with $c\in C_Q$ and $n\in N$, which give $\theta(g)=\psi_Q(c)\chi(n)$.
Similarly, $1<Q^g\leq\Omega_1(S\cap\ls hS)$, and we can write $h=dm$ with $d\in C_{Q^g}$ and $m\in N$. 
Then $\theta(h)=\psi_{Q^g}(d)\chi(m)$.
Now, $gh=g(dm)=\ls gd(gm)=(\ls gdc)(nm)$, with $\ls gdc\in C_Q$ and $nm\in N$.
We calculate $\theta(gh)=\psi_Q(\ls gdc)\chi(nm)=\psi_Q(c)\chi(n)\psi_Q(\ls gd)\chi(m)$, using that $\psi_Q$ and $\chi$ are group homomorphisms and that $k^\times$ is abelian.
It remains to check that $\psi_Q(\ls gd)=\conj(g)\psi_{Q^g}(d)$. 
First, note that $\conj(g)\psi_{Q^g}(d)=\conj(n)\psi_{Q^g}(d)$, since $g=cn$ and $c\in C_Q$. 
Moreover, $\conj(n)\psi_{Q^g}|_{N_S\cap C_Q}=\chi|_{N\cap C_Q}=\psi_Q|_{N\cap C_Q}$.
By uniqueness of the extensions of $\chi$, stated in the bijection $(*)$ above, we conclude that $\conj(g)\psi_{Q^g}=\psi_Q$. 
Therefore, $\psi_Q(\ls gd)=\psi_{Q^g}(d)$, and $\theta(gh)=\psi_Q(c)\chi(n)\psi_{Q^g}(d)\chi(m)=\theta(g)\theta(h)$.
\end{itemize}

\end{proof}

\section{Finite groups with metacyclic Sylow $p$-subgroups, $p$ odd}\label{sec:mc}

A thorough description of finite metacyclic $p$-groups can be found in \cite[Section III.11]{huppert}. In this section, we present the technical facts we need about them.

\begin{defn}\label{def:mc}
A finite group $G$ is {\em metacyclic} if $G$ is an extension
$$\xymatrix{1\ar[r]&C_m\ar[r]&G\ar[r]&C_n\ar[r]&1}$$
for some positive integers $m,n$. 
We call $G$ {\em split metacyclic} if the extension splits.
\end{defn}

Henceforth, we let $p$ denote an odd prime and let
$$S=\langle x,y\mid x^{p^m}=1,~y^{p^n}=x^{p^q},~\ls yx=x^{1+p^l}\rangle=\langle x,y\rangle(p^m,p^n,p^l+1,p^q)$$
be a noncyclic metacyclic $p$-group of order $p^{m+n}$, where $1\leq l,m,n,q$ are integers such that $l,q\leq m\leq l+n$, 
and they are subject to the congruences $(1+p^l)^{p^n}\equiv1\pmod{p^m}$, and $p^{q+l}\equiv0\pmod{p^m}$.
In \cite[Proposition 3.1]{Di}, J. Dietz shows that $S$ is split if and only if we can choose generators $x,y$ of $S$ such that the above relations hold with $q=m$. 
Otherwise, $q<m$ and $l<q<n$.
If $S$ is split, we write $S=\langle x,y\rangle(p^m,p^n,p^l+1)$ instead of $\langle x,y\rangle(p^m,p^n,p^l+1,p^m)$. 
Note that $S$ is abelian if and only if $l=m$.

The following facts about finite metacyclic $p$-groups are well known.
\begin{lemma}\label{lem:mc}
Assume the above notation (and recall that $p$ is odd).
\begin{enumerate}
\item Every subquotient of a metacyclic group is metacyclic. 
\item $S$ has a unique elementary abelian subgroup $E=\Omega_1(S)$ of rank $2$. In particular, $E,C_S(E)$ are characteristic subgroups of $S$, with $|S:C_S(E)|\in\{1,p\}$.
\item If $E\leq Z(S)$, then $S$ has nilpotency class at most $2$.
\item If $E\not\leq Z(S)$, let $Z=E\cap Z(S)$. Then $Z\car S$ has order $p$ and, for any $g\in G$ such that $\ls gZ\leq S$, then $Z=\ls gZ$, i.e. $g\in N_G(Z)$.
\end{enumerate}
\end{lemma}

\begin{rem}
If $p=2$, then $S$ need not have a unique elementary abelian subgroup of rank $2$. Indeed, nonabelian dihedral $2$-groups are metacyclic and possess two conjugacy classes of Klein four groups. 
\end{rem} 

\begin{proof} 
\begin{enumerate}
\item Let $P\leq S$. Set $U=\langle x\rangle\cap P$. Then $U\trianglelefteq P$ and $P/U\cong P\langle x\rangle/\langle x\rangle$ is isomorphic to a subgroup of $\langle y\rangle$.
\item This is proved in \cite[Lemma 2.1]{Ma-mc}. The `In particular' statement is immediate from $E=\Omega_1(S)$.
\item Let $E$ be the unique elementary abelian subgroup of $S$. Since $S'=\langle x^{p^l}\rangle$ is cyclic, the inclusion $E\leq Z(S)$ shows that $Z(S)\not\leq S'$. 
By a result of P. Hall \cite[(2.46)]{hall}, it follows that $S$ has class at most $2$.
\item Since $1<E\car S$ and $E\not\leq Z(S)$, we have $1<E\cap Z(S)<E$, with $Z=E\cap Z(S)$ a characteristic subgroup of $S$ of order $p$.
Let $g\in G$ be such that $\ls gZ\leq S$. 
Since $N_G(S)$ controls $p$-fusion in $G$ by Theorem~\ref{thm:stancu}, there exist $h\in N_G(S)$ and $c\in C_G(Z)$ such that $g=hc\in N_G(S)C_G(Z)=N_G(Z)$. 

\end{enumerate}
\end{proof}

Routine computations using the defining relations yield the following technical results which we will need.

\begin{lemma}\label{lem:mc-rels}
Let $S=\langle x,y\rangle(p^m,p^n,p^l+1,p^q)$ be a metacyclic $p$-group, for an odd prime $p$.
\begin{enumerate}
\item $Z(S)=\langle x^{p^{m-l}},y^{p^{m-l}}\rangle$ and $E=\langle x^{p^{m-1}},x^{-p^{q-1}}y^{p^{n-1}}\rangle$. In particular, if $S$ is split, then $E=\Omega_1(S)=\langle x^{p^{m-1}},y^{p^{n-1}}\rangle$.
\item $E\not\leq Z(S)$ if and only if $Z(S)$ is cyclic if and only if $m=n+l$. 
Moreover, in this case, the $p$ noncentral subgroups of $S$ of order $p$ are $S$-conjugate, and their centraliser in $S$ is the characteristic subgroup $C_S(E)=\langle x^p,y\rangle$ of index $p$.
\end{enumerate}
\end{lemma}

\begin{proof}
Suppose that $S$ is not abelian, i.e. $l<m$.
We have $x^{p^{m-1}}\in Z(S)$ since $\langle x\rangle\trianglelefteq S$.
Write $Z(S)=\langle x^a,y^b\rangle$, where $1\leq a\leq p^{m-1}$ and $1\leq b\leq p^n$ are the smallest positive integers such that 
\begin{align*}
y&=\ls{x^a}y=x^a(\ls yx)^{-a}y=x^{-ap^l}y\\
x&=\ls{y^b}x=x^{(1+p^l)^b}\qbox{with}(1+p^l)^b\equiv1+bp^l\pmod{p^{2l}}.
\end{align*}
From these, we conclude that $a,b\equiv0\pmod{p^{m-l}}$.

The computation of $E$ is proved in \cite[Lemma 2.1]{Ma-mc} (using $s=1$, which is possible by a suitable choice of generators).

Now, $E\not\leq Z(S)$ if and only if $Z(S)$ is cyclic and $y^{p^{n-1}}$ does not commute with $x$. 
From the previous computation, $[y^{p^{n-1}},x]\neq1$ if and only if $n-1<m-l\leq n$, since $m\leq l+n$. Therefore $E\not\leq Z(S)$ if and only if $m=n+l$.

Finally, if $C_S(E)\neq S=N_S(E)$, then conjugation by some element of $S$ induces a nontrivial automorphism of $E$ of order $p$, which must be a permutation of the noncentral subgroups of order $p$.
If $m=n+l$, we obtain\; 
$$[y^{p^{n-1}},x^p]=1\qbox{and}[y^{p^{n-1}},x]=x^{(1+p^l)^{p^{n-1}}-1}=x^{p^{m-1}}.$$
The result follows given that $x^{p^{m-1}}\in Z(S)$.
\end{proof}

We now highlight some properties of finite groups with a metacyclic Sylow $p$-subgroup.

\begin{defn}\label{def:control-fusion}
Let $G$ be a finite group and let $S\in\Syl_p(G)$. 
A subgroup $H$ of $G$ {\em controls $p$-fusion in $G$} if, for any subgroup $Q$ of $S$ and any element $g\in G$ such that $\ls gQ\leq S$, there exist $h\in H$ and $c\in C_G(Q)$ such that $g=hc$.
In other words, the conjugation maps $\conj(g),\conj(h):Q\to\ls gQ=\ls hQ$ are equal.
\end{defn}
Equivalently, by switching sides, $H$ controls $p$-fusion in $G$ if and only if, for any $g\in G$ such that $Q^g\leq S$, there exist $c\in C_G(Q)$ and $h\in H$ such that $g=ch$.

\begin{thm}\cite[Th\'eor\`eme 3.3.1 and D\'efinition 3.1.1]{St}\label{thm:stancu}
Let $p$ be an odd prime and let $G$ be a finite group of order divisible by $p$. 
Suppose that a Sylow $p$-subgroup $S$ of $G$ is metacyclic. Then $N_G(S)$ controls $p$-fusion in $G$. 
\end{thm}

\begin{rem}\label{rem:reduction-split}
By \cite[Theorem 1.3, Note]{Di}, a finite group $G$ with a nonabelian nonsplit metacyclic Sylow $p$-subgroup $S$ is $p$-nilpotent; that is, $G=O_{p'}(G)\rtimes S$.
In this case, $S$ controls $p$-fusion in $G$, and as a consequence, every $p$-local subgroup of $G$ has the form $N_G(Q)=O_{p'}(C_G(Q))\rtimes S_Q$, where $S_Q\in\Syl_p(N_G(Q))$.
\end{rem}

\vspace{.3cm}
Let $S=\langle x,y\rangle(p^m,p^n,p^l+1,p^q)\in\Syl_p(G)$. Then $\Aut_G(S)=N_G(S)/C_G(S)$ is isomorphic to a subgroup of $\Aut(S)$. Suppose that $p$ is odd.
If $S$ is nonsplit metacyclic, then Remark~\ref{rem:reduction-split} implies that $\Aut_G(S)=\operatorname{Inn}_G(S)$ is a $p$-group, i.e. $N_G(S)/C_G(S)\cong S/Z(S)$. 
If $S$ is split, then the $p'$-part of $\Aut(S)$ is cyclic, as stated in \cite[Proposition 3.2]{Di}. 
Its proof uses \cite[Theorem 5.1.4 (Burnside)]{G}, which states that any nonidentity $p'$-automorphism of $S$ induces a nonidentity automorphism of the Frattini quotient $S/\Phi(S)$. 
It follows that, conjugation by any $p'$-element of $N_G(S)$ on $S$, induces an automorphism of the $2$-dimensional $\F_p$-vector space $S/\Phi(S)$, and so can be expressed as a matrix in $\GL_2(p)$. 

For later use, we expand a little on the computations for a split metacyclic $p$-group of the form $S=\langle x,y\rangle(p^m,p^n,p^l+1)$, which also helps us set the notation. We start with the power rule: for any positive integers $a,c,\alpha$,
\begin{equation}\label{eqn:basic}
(x^ay^c)^\alpha=x^{a\sum_{i=0}^{\alpha-1}(1+p^l)^{ic}}y^{\alpha c}\quad\qbox{with}\quad\sum_{i=0}^{\alpha-1}(1+p^l)^{ic}\equiv\alpha+\frac{\alpha(\alpha-1)c}2p^l\pmod{p^{2l}}
\end{equation}
 
Let $\varphi\in\Aut(S)$, say $\varphi(x)=x^ay^c$ and $\varphi(y)=x^by^d$ with $0\leq a,b<p^m$ and $0\leq c,d<p^n$, and consider the induced map $\ovl{\varphi}$ as a matrix in $\Aut(S/\Phi(S))\cong\GL_2(p)$.
Let $x=\begin{pmatrix}1\\0\end{pmatrix}$ and $y=\begin{pmatrix}0\\1\end{pmatrix}$, so that $\ovl{\varphi}=\begin{pmatrix}\bar a&\bar b\\\bar c&\bar d\end{pmatrix}\in\GL_2(p)\cong\Aut(S/\Phi(S))$, where a bar indicates reduction modulo $p$, and the action is on the left.

The relations of $S$ and Equation~\eqref{eqn:basic} yield the equalities
\begin{align}\label{x-power}
1&=(x^ay^c)^{p^m}=x^{a\sum_{i=0}^{p^m-1}(1+p^l)^{ic}}y^{cp^m}=y^{cp^m}\qbox{since}\;\sum_{i=0}^{p^m-1}(1+p^l)^{ic}\equiv0\pmod{p^m}\\\label{y-power}
1&=(x^by^d)^{p^n}=x^{b\sum_{i=0}^{p^n-1}(1+p^l)^{id}}y^{dp^n}=x^{bp^n}\qbox{using that}\;\sum_{i=0}^{p^n-1}(1+p^l)^{id}\equiv p^n\pmod{p^{l+n}}
\end{align}
(with $m\leq l+n$) 
\begin{equation}\label{eq:conj}
(x^by^d)(x^ay^c)(y^{-d}x^{-b})=(x^ay^c)^{1+p^l}\qbox{giving}x^{a(1+p^l)^d+b-b(1+p^l)^c}y^c=x^{a\sum_{i=0}^{p^l}(1+p^l)^{ic}}y^{c(1+p^l)}
\end{equation}
and hence
$$a(1+p^l)^d+b-b(1+p^l)^c\equiv a\sum_{i=0}^{p^l}(1+p^l)^{ic}\pmod{p^m}\qbox{and}c\equiv c(1+p^l)\pmod{p^n}.$$
The exponents modulo $p^{2l}$ and $p^n$ respectively in Equation~\eqref{eq:conj} are 
\begin{equation}\label{eq:y-conj}
a+(ad-bc)p^l\equiv a+ap^l\pmod{p^{2l}}\qbox{and}cp^l\equiv0\pmod{p^n}.
\end{equation}

From these, we record the following constraints:
\begin{itemize}
\item If $n>l$, then $c\equiv0\pmod{p^{n-l}}$, and $b$ is independent. In \eqref{eq:y-conj} taken modulo $p$, we obtain $d\equiv1\pmod p$ and $a\not\equiv0\pmod p$. 
In other words, $\ovl{\varphi}=\begin{pmatrix}\bar a&\bar b\\\bar 0&\bar 1\end{pmatrix}\in\GL_2(p)$.
\item If $n\leq l<m$, then $n\leq m\leq l+n\leq2l$, $c$ is independent and $b\equiv0\pmod{p^{m-n}}$. We obtain $d\equiv1\pmod p$ and $a\not\equiv0\pmod p$.
In other words, $\ovl{\varphi}=\begin{pmatrix}\bar a&\bar 0\\\bar c&\bar 1\end{pmatrix}\in\GL_2(p)$.
\end{itemize}

Observe that, if $l=m=n$, the above computations do not yield any constraints on the values of $a,b,c,d$, except that $ad-bc\not\equiv0\pmod p$.

Combining these computations with the fact that $N_G(S)$ controls $p$-fusion in $G$, we deduce the following.
\begin{lemma}\label{lem:aut-mc}
Let $G$ be a finite group and let $S\in\Syl_p(G)$ for an odd prime $p$. Suppose that $S$ is split metacyclic.
Then there exists $w\in G$ such that $N_G(S)$ can be written as as extension
$\xymatrix{1\ar[r]&SC_G(S)\ar[r]&N_G(S)\ar[r]&\langle w\rangle\ar[r]&1}$, where $w^{|A|}\in C_G(S)$.
Hence, $\Aut_G(Q)=O_p(\Aut_G(Q))A$, where $A$ is a cyclic group of order dividing $p-1$ for every subgroup $Q$ of $S$. 
In particular, $\Out_G(S)$ is cyclic of order dividing $p-1$, and for every $Q\car S$, the inclusion $N_G(S)\hookrightarrow N_G(Q)$ induces a surjection $\Aut_G(S)\to\Aut_G(Q)$.
\end{lemma}

For this last statement, note that if $1<Q\car S$, then the inclusion $C_G(S)\leq N_G(S)\cap C_G(Q)$ yields a surjective group homomorphism
$$\xymatrix{N_G(S)/C_G(S)\ar@{->>}[r]& N_G(S)/\big(N_G(S)\cap C_G(Q)\big)}\cong N_G(S)C_G(Q)/C_G(Q)=N_G(Q)/C_G(Q).$$

\begin{rem}\label{rem:aut-mc}
In Section~\ref{sec:et-mc}, we will need a more detailed information about the $p$-local subgroups of a finite group $G$ with a split metacyclic Sylow $p$-subgroup $S$ such that $E\not\leq Z(S)$.
Specifically, we want to describe $N_G(Q)$ for $1<Q\leq E$.

Let $z=x^{p^{m-1}},\;u=y^{p^{n-1}}$, and write $E=\langle u,z\rangle,\;Q_j=\langle uz^j\rangle\quad(0\leq j<p)$ and $Q_p=Z=\langle z\rangle$. 
Note that $E$ and $Z=\Omega_1(Z(S))=E\cap Z(S)$ are characteristic subgroups of $S$.
We have $C_S(Q_j)=C_S(E)=\langle x^p,y\rangle\car S$ and $\ls xQ_j=Q_{j+1}$ for all $0\leq j<p$, where the index is taken modulo $p$.
For convenience, let $S_E=C_S(E)$ and $N=N_G(S)=SC_G(S)\langle w\rangle$ for some $w\in N$, as described in Lemma~\ref{lem:aut-mc}. 
Since $N_G(Q)=(N\cap N_G(Q))C_G(Q)$ with equality if $Q\car S$, and since $C_G(S)\leq C_G(Q)$, conjugation by $g\in N_G(Q)$ on $Q$ can be realised by an element of $S\langle w\rangle$.

Write $\ls wx=x^ay^c$ and $\ls wy=x^by^d$. Then the above relations \eqref{x-power}--\eqref{eq:y-conj} with $1\leq l,n<m=l+n$ yield
$$b\equiv0\pmod{p^l}\;,\hspace{.8cm}c\equiv0\pmod{p^{n-l}}\hspace{.8cm}\hbox{and}\hspace{.8cm}d=1,$$
since Equation~\eqref{eq:y-conj} gives $d\equiv1\pmod{p^{m-l}}$ and $n=m-l$.
It follows that $\ls wz=z^a$, and $\ls wu=u$. 
That is, $\ls wQ_j=Q_{j+a}$ for $1\leq j\leq p-1$ and $w\in N_G(Z)\cap C_G(Q_0)$. 
Since $\ls xQ_j=Q_{j+1}$, it follow that $N_G(Q_j)=C_G(Q_j)=\ls{x^j}C_G(Q_0)\geq\langle\ls{x^j}w\rangle$ for all $0\leq j<p$. 
Moreover, $S_EC_G(S)\langle w\rangle\leq C_G(Q_0)$, whilst $N_G(Z)=C_G(Z)\langle w\rangle\neq C_G(Z)$.

To summarise, we have 
$$N_G(P)=\left\{\begin{array}{rll}
C_G(E)S\langle w\rangle&\geq C_G(S)S\langle w\rangle&\qbox{if $P=E$}\\
C_G(Z)\langle w\rangle&\geq C_G(S)S\langle w\rangle&\qbox{if $P=Z$}\\
C_G(P)&\geq S_EC_G(S)\langle\ls{x^i}w\rangle&\qbox{if $P=Q_i$ for $0\leq i<p$.}
\end{array}\right.$$

\end{rem}


\section{Orbit category and endotrivial modules in the metacyclic case}\label{sec:et-mc}

Throughout this section, we assume the following.
\begin{nota}\label{nota:et}
Let $G$ be a finite group and let $p$ be an odd prime dividing $|G|$ such that $S\in\Syl_p(G)$ is metacyclic noncyclic.
Write $S=\langle x,y\rangle(p^m,p^n,1+p^l,p^q)$ with $1\leq l,m,n,q$ and $q,l\leq m\leq m+n$, and write $E=\Omega_1(S)$.
\end{nota}
Without loss of generality, we may assume that $G$ has no proper strongly $p$-embedded subgroup, i.e. $G=G_0=\langle N_G(Q)\mid1<Q\leq S\rangle$, since the restriction $\Res^G_{G_0}:T(G)\to T(G_0)$ is an isomorphism.
That is:
\begin{itemize}
\item If $S$ is nonsplit nonabelian and $p$ is odd, then \cite[Theorem 1.3, Note]{Di} shows that $G$ is $p$-nilpotent, and therefore its subgroups too and $S$ controls $p$-fusion in $G$. 
Then, $N_G(Q)=O_{p'}(C_G(Q))\rtimes S_Q$, where $S_Q\in\Syl_p(N_G(Q))$ for all $1<Q\leq S$, and $G=\langle S,O_{p'}(C_G(Q))\mid1<Q\leq E\rangle$, where it suffices to take one subgroup per conjugacy class.
\item If $S$ is cyclic (i.e. nonsplit abelian), then $G=N_G(\Omega_1(S))$, as recalled above and in \cite{MT}.
\item If $S$ is split metacyclic, then $N_G(Q)\leq N_G(\Omega_1(Q))=(N_G(S)\cap N_G(\Omega_1(Q)))C_G(\Omega_1(Q))$ for all $1<Q\leq S$. 
Therefore, $G=\langle w, C_G(Q)\mid1<Q\leq E\rangle$, for $w$ as in Remark~\ref{rem:aut-mc}.
Note that, if $E\not\leq Z(S)$, then $G=\langle C_G(Q),~C_G(Z)\rangle$, by Lemmas~\ref{lem:mc} and~\ref{lem:mc-rels}.
\end{itemize}

We want to determine $\pi_1(\CO^*_p(G))$.
If $S$ is cyclic, then \cite[Corollary 4.14]{Gr} shows that $\pi_1(\CO^*_p(G))\cong G/O^{p'}(G)$, assuming that $G=N_G(E)$. 

If $S$ is nonabelian nonsplit, then $G$ is $p$-nilpotent and if $G$ has no proper strongly $p$-embedded subgroup, then $H_1(\CO^*_p(G))\cong G/G'S$, cf. \cite[Theorem E and Section 6.4]{Gr}.
Thus, it remains to consider the case when $S$ is split metacyclic. In view of Proposition~\ref{prop:omega1central}, we distinguish two cases depending on whether $E\leq Z(S)$ or not.

\begin{prop}\label{prop:main}
Let $p$ be an odd prime, let $G$ be a finite group with order divisible by $p$ and let $S\in\Syl_p(G)$. Suppose that $S$ is split metacyclic and that $G$ has no strongly $p$-embedded subgroup. Let $E=\Omega_1(S)$.
\begin{enumerate}
\item If $E\leq Z(S)$, then $\pi_1(\CO^*_p(G))\cong N_G(S)/\langle N_G(S)\cap O^{p'}(N_G(P))\mid1<P\leq E\rangle$.
\item If $E\not\leq Z(S)$, then $\pi_1(\CO^*_p(G))\cong N_G(S)/R$, where $R$ is generated by the elements of $N_G(S)\cap O^{p'}(N_G(Z))$ and those of $N_G(S)$ that can be written as a product of the form 
$g=hh't$, with with $h\in O^{p'}(N_G(E))\cap C_G(Q)$, $h'\in O^{p'}(C_G(Q))$ and $t\in S$. 
\end{enumerate}
\end{prop}

\begin{proof}
 If $E\leq Z(S)$, the statement holds by Theorem~\ref{thm:main}. So it remains to show the proposition in the case when $E$ is not a central subgroup of $S$.
 
Suppose $E\not\leq Z(S)$. 
We use the notation of Remark~\ref{rem:aut-mc} and set $S=\langle x,y\rangle(p^m,p^n,1+p^l)$, $Z=E\cap Z(S)=\langle x^{p^{m-1}}\rangle$, $Q=Q_0=\langle\;u=y^{p^{n-1}}\rangle$, and $C_S(E)=S_E$.
Picking $w\in G$ as in Remark~\ref{rem:aut-mc}, we have $N_G(S)=SC_G(S)\langle w\rangle$, $N_G(Z)=N_G(S)C_G(Z)=C_G(Z)\langle w\rangle$, and $N_G(Q)=C_G(Q)\geq C_G(E)\langle w\rangle$.
Recall that $E$ and $Z$ are both strongly closed in $S$ with respect to $G$, i.e. no element of $E$ is $G$-conjugate to an element of $S$ outside $E$, and similarly for $Z$. 

Note that $O^{p'}(N_G(E))\leq O^{p'}(N_G(Z))=\langle\ls gS\mid g\in[C_G(Z)/SC_G(S)]\rangle$, since $SC_G(S)=N_G(S)\cap C_G(Z)$, and similarly
$O^{p'}(N_G(Q))=\langle\ls gS_E\mid g\in[C_G(Q)/S_EC_G(S_E)\langle w\rangle]\rangle$, where $N_G(S_E)\cap C_G(Q)=S_EC_G(S_E)\langle w\rangle$ has index $p$ in $N_G(S_E)$.

Let $R$ be the subgroup of $N_G(S)$ defined in Theorem~\ref{thm:gr}.
That is, $R$ is generated by all the elements $g\in N_G(S)$ such that there exist $1<P_0,\dots,P_n\leq S$ and elements $g_i\in O^{p'}(N_G(P_i))$ for all $1\leq i\leq n$ with
\begin{itemize}
\item $g=g_1\cdots g_n$, and 
\item $P_0^{g_1\cdots g_i}\leq P_{i+1}$, for all $1\leq i\leq n-1$.
\end{itemize}

By Remark~\ref{rem:rho-sgps}, it suffices to consider the nontrivial subgroups of $E$ to determine $R$.
Let $g\in R$, with factorisation $g=g_1\cdots g_n$  as above. 

If $P_0=E$ or if $P_0=Z$, then we can take $n=1$ and $P_1=P_0$ since $E$ and $Z$ are both strongly closed in $S$ with respect to $G$ and $N_G(Z)\geq N_G(E)$. 
It follows that in these cases $g\in N_G(S)\cap O^{p'}(N_G(Z))$.

If $P_0=Q$, then $P_1=Q$ or $P_1=E$. 
By taking $g_i=1$ if necessary, we may then assume that we have a sequence $P_0=P_1=Q$, $P_{2i}=E$ and $P_{2i+1}=Q^{g_1\cdots g_{2i}}$ for all $i\geq1$. 
(Note also that $P_1=Q^{g_1}=Q,\;P_3=Q^{g_1g_2g_3}=Q^{g_2}$, etc, giving in fact $P_{2i+1}=Q^{g_2\cdots g_{2i}}$.)

Since the noncentral subgroups of $E$ are $S$-conjugate and $N_G(Q)=C_G(Q)$, if $g\in O^{p'}(N_G(E))$, then $Q,Q^g\leq E$ and so, there exist $t\in S$ and $h\in C_G(Q)$ such that $g=ht$.
From the inclusion $S\leq O^{p'}(N_G(E))$ it follows that $h\in O^{p'}(N_G(E))\cap C_G(Q)$.
Proceeding stepwise, we obtain that $g_1=h_1\in O^{p'}(C_G(Q))$, then $g_2=h_2t_2$ with $h_2\in O^{p'}(N_G(E))\cap C_G(Q)$ and $t_2\in S$, giving $Q^{g_1g_2}=Q^{t_2}$.
Next, we have $g_3=h_3^{t_2}\in O^{p'}(C_G(Q^{t_2}))=O^{p'}(C_G(Q))^{t_2}$, and then $g_4\in O^{p'}(N_G(E))\cap C_G(Q)^{t_2}S$, say $g_4=h_4^{t_2}t_4$ with $h_4\in O^{p'}(N_G(E))\cap C_G(Q)$ and $t_4\in S$.
To summarise, the factorisation of $g$ starts with 
$$g_1g_2g_3g_4=h_1h_2t_2h_3^{t_2}h_4^{t_2}t_4=\underbrace{h_1h_2h_3h_4}_{\in C_G(Q)}\underbrace{t_2t_4}_{\in S}.$$
Iteratively, $g=g_1\cdots g_n=h_1\cdots h_nt$, with $h_{2i-1}\in O^{p'}(C_G(Q))$, $h_{2i}\in O^{p'}(N_G(E))\cap C_G(Q)$ and $t\in S$, such that $g\in N_G(S)\cap C_G(Q)S=SC_G(S)\langle w\rangle$.

Since $O^{p'}(C_G(Q))$ is normalised by $O^{p'}(N_G(E))\cap C_G(Q)$, we can write $h_1\cdots h_nt=hh't$ with $h\in O^{p'}(N_G(E))\cap C_G(Q)$, $h'\in O^{p'}(C_G(Q))$ and $t\in S$. Thus $g=hh't$.

\end{proof}

\begin{rem}
As concluding remark, let us emphasise the point observed above, that, in the assumptions of Proposition~\ref{prop:main}, the case when $E$ is not central in $S$ implies that $S$ has a cyclic centre. Thus, the second case of the proposition concerns a very particular type of metacyclic $p$-groups.
\end{rem}

%
%
%
%
%


\end{document}